\documentclass[11pt]{amsart}
\usepackage{amsfonts}
\usepackage{amscd}
\usepackage{amssymb}
\usepackage{graphics}
\usepackage{amsmath}
\usepackage[english]{babel}
\usepackage{hyperref}

\hypersetup{colorlinks,citecolor=blue}

\newcommand{\BN}{{\mathbb{N}}}
\newcommand{\BR}{{\mathbb{R}}}
\newcommand{\BC}{{\mathbb{C}}}

\newcommand{\BQ}{{\mathbb{Q}}}

\newcommand{\BH}{{\mathbb{H}}}

\newcommand{\OO}{{\mathcal{O}}}

\newcommand{\gD}{\Delta}
\newcommand{\gd}{\delta}
\newcommand{\gb}{\beta}

\newcommand{\gC}{\Gamma}
\newcommand{\gc}{\gamma}
\newcommand{\gs}{\sigma}
\newcommand{\gS}{\Sigma}

\newcommand{\go}{\omega}
\newcommand{\gep}{\epsilon}

\newcommand{\gL}{\Lambda}
\newcommand{\ga}{\alpha}

\newcommand{\ti}[1]{\tilde{#1}}

\newcommand{\ram}{\mathrm{Ram}}

\newcommand{\vol}{\mathrm{vol}}

\newcommand{\Min}{\mathrm{Min}}

\newcommand{\SL}{\mathrm{SL}}
\newcommand{\GL}{\mathrm{GL}}
\newcommand{\PSL}{\mathrm{PSL}}
\newcommand{\PGL}{\mathrm{PGL}}
\newcommand{\SO}{\mathrm{SO}}
\newcommand{\PSO}{\mathrm{PSO}}

\newcommand{\covol}{\mathrm{covol}}
\newcommand{\Irr}{\mathrm{Irr}}

\newtheorem{prop}{Proposition}[section]
\newtheorem{thm}[prop]{Theorem}
\newtheorem{lem}[prop]{Lemma}
\newtheorem{cor}[prop]{Corollary}

\theoremstyle{definition}

\newtheorem{rem}[prop]{Remark}

\def\c{\chi}

  \def\G{\Gamma}

  \def\la{\lambda}

  \def\s{\sigma}

  \def\go{\rightarrow}

\begin{document}

\author[M. Belolipetsky]{Mikhail Belolipetsky}
\address{Department of Mathematical Sciences, Durham University, Durham DH1 3LE, UK;
Sobolev Institute of Mathematics, Koptyuga 4, 630090 Novosibirsk, Russia}
\curraddr{}
\email{mikhail.belolipetsky@durham.ac.uk}
\thanks{The authors acknowledge support from the BSF, ISF, EPSRC and ERC, and the valuable comments and corrections of the anonymous referee}

\address{Institute of Mathematics, The Hebrew University, Givat Ram, Jerusalem, 91904, Israel}
\author[T. Gelander]{Tsachik Gelander}
\curraddr{}
\email{tsachik.gelander@gmail.com}
\thanks{}
\author[A. Lubotzky]{Alex Lubotzky}
\curraddr{}
\email{alexlub@math.huji.ac.il}
\thanks{}
\author[A. Shalev]{Aner Shalev}
\curraddr{}
\email{shalev@math.huji.ac.il}
\thanks{}

\date{\today}

\title{Counting Arithmetic Lattices and Surfaces}
\maketitle

\begin{abstract}
We give estimates on the number $\mathrm{AL}_H(x)$ of conjugacy classes of arithmetic lattices
$\Gamma$ of covolume at most $x$ in a simple Lie group $H$. In particular,
we obtain a first concrete estimate on the number of arithmetic $3$-manifolds of volume at most $x$.
Our main result is for the classical case $H=\PSL(2,\BR)$ where we show that
$$
 \lim_{x\to\infty}\frac{\log \mathrm{AL}_H(x)}{x\log x}=\frac{1}{2\pi}.
$$
The proofs use several different techniques: geometric (bounding the
number of generators of $\Gamma$ as a function of its covolume),
number theoretic (bounding the number of maximal such $\Gamma$)
and sharp estimates on the character values of the symmetric groups
(to bound the subgroup growth of $\Gamma$).
\end{abstract}

%-----------------------------------------------------------------------------------------

\section{Introduction}

Let $H$ be a noncompact simple Lie group with a fixed Haar measure
$\mu$. A discrete subgroup $\Gamma$ of $H$ is called a lattice if
$\mu({\Gamma}\backslash{H})<\infty$. A classical theorem of Wang \cite{Wang}
asserts that if $H$ is not locally isomorphic to $\PSL_2(\BR)$ or
$\PSL_2(\BC)$, then for every $0<x\in\BR$ the number
$\mathrm{L}_H(x)$ of conjugacy classes of lattices in $H$ of covolume
at most $x$ is finite. This result was greatly extended by Borel
and Prasad \cite{BP}. In recent years there has been an attempt to
quantify Wang's theorem and to give some estimates on
$\mathrm{L}_H(x)$ (see \cite{BGLM, Ge, GLNP, Be, BL}).

If $H=\PSL_2(\BR)$ or $\PSL_2(\BC)$, then $\mathrm{L}_H(x)$ is
usually not finite (and even uncountable in the first case). Still
Borel \cite{Bo} showed that the number $\mathrm{AL}_H(x)$ of conjugacy classes
of \emph{arithmetic} lattices in $H$ of covolume at most $x$ is finite
for every $x\in\BR$.

In this paper we study the asymptotic behavior of AL$_H(x)$ when
$x \to \infty$. Our first result gives a general upper bound.

\begin{thm}\label{thm:upperbound}
Assume that $H$ is of real rank one.
There exists a constant $b = b(H, \mu)$ such that
$\mathrm{AL}_H(x) \le x^{bx}$ for all $x \gg 0$.
\end{thm}

Theorem \ref{thm:upperbound} is also true if the rank of $H$ is greater than one, see Remark \ref{rem:5.1} below and \cite{Ge:morse}, but for most higher rank groups much better estimates are given in \cite{BL}. Our next result shows that Theorem~\ref{thm:upperbound} is best
possible in general.
\medskip

\noindent
\begin{thm}\label{thm:lowerbound} For $H = \mathrm{PSO}(n,1)$ there exists a constant
$a = a(n) > 0$ such that $\mathrm{AL}_H(x) \ge x^{ax}$ for all $x \gg 0$.
\end{thm}

One novelty of the current work compared to \cite{BGLM} and \cite{Ge} is that it deals with orbifolds rather than manifolds, i.e. we do not require the lattices to be torsion free. However, it is clear from the proof that the lower bound remains valid when restricting only to conjugacy classes of arithmetic {\it torsion free} lattices.
Another novelty is that it covers the case of $H=\mathrm{PSO}(3,1)=\mathrm{Isom}(\BH^3)$, for which the result translates to: {\it the number of arithmetic hyperbolic $3$ orbifolds (or manifolds) of volume at most $x$ is roughly $x^{cx}$ for large $x$}. Prior to this work no explicit upper bound was known in this case, as well as for the case $H=\PSO(2,1)$. The upper bound obtained here confirms the expected estimate which follows from the Lehmer conjecture concerning algebraic integers (cf. \cite{Ge}).

The proofs of Theorems \ref{thm:upperbound} and \ref{thm:lowerbound} allow one to compute concrete constants, but in general it seems very difficult to obtain sharp estimates for these constants. The main part of this paper is dedicated to the classical case $H=\SO(2,1)^{\circ} \cong \PSL_2(\BR)$
where we obtain a very sharp estimate:

\begin{thm}\label{thm:SL(2,R)}
Let $H=\PSL_2(\BR)$ endowed with the Haar measure induced from the
Riemanian measure of the hyperbolic plane $\BH^2= \PSL_2(\BR)/\PSO(2)$.
Then
 $$\lim_{x\to\infty}\frac{\log \mathrm{AL}_H(x)}{x\log x}=\frac{1}{2\pi}.$$
\end{thm}

The proof of Theorem \ref{thm:SL(2,R)} shows:

\begin{cor}\label{cor:1.4}
Let $\mathrm{AS}(g)$ be the number of arithmetic Riemann surfaces of genus $g$. Then
$$
 \lim_{g\to\infty}\frac{\log \mathrm{AS}(g)}{g\log g}=2.
$$
\end{cor}

Let us now describe the main ingredients of the proofs. We start with a
result on the number $d(\G)$ of generators of lattices $\G$,
which is of independent interest.

\begin{thm}\label{thm:d(Gamma)}
Let $H$ be a connected simple Lie group of real rank one. Then there is an effectively  computable constant $C=C(H)$ such that for any lattice $\gC < H$ we have $d(\gC)\le C\cdot\vol(\gC\backslash H)$.
\end{thm}

The proof of Theorem \ref{thm:d(Gamma)} is geometric and valid for all
lattices, not necessarily arithmetic. Note that Theorem \ref{thm:d(Gamma)} implies
the celebrated Kazhdan--Margulis theorem \cite{KM} asserting that there is a
common lower bound on the covolumes of all lattices in $H$. Indeed, $d(\gC)\ge 2\Rightarrow \vol(\gC\backslash H)\ge\frac{2}{C}$. It also has several other applications, for instance, it gives a linear bound on the first Betti number of orbifolds in terms of their volume (cf. \cite{Fi-Gr} and see Remark \ref{rem:d(Gamma)} below and \cite{Ge:morse}).

Another essential component in our proofs is the following.

\begin{thm}\label{thm:counting-maximal}\label{thm:2.1}
Let $\mathrm{MAL}_H^u(x)$ (resp. $\mathrm{MAL}_H^{nu}(x)$) denote the
number of conjugacy classes of maximal uniform (resp. non-uniform)
irreducible arithmetic lattices of covolume at most $x$ in $H=\PGL_2(\BR)^a\times\PGL_2(\BC)^b$.
Then:
\begin{itemize}
 \item[(i)] There exists a positive constant $\ga=\ga(a,b)$, and for every $\epsilon>0$ a positive constant $\gb=\gb(\epsilon,a,b)$, such that
 $$
  x^\ga\leq \mathrm{MAL}_H^u(x)\leq x^{\gb(\log x)^\epsilon},~\text{for}~x\gg 0.
 $$
 \item[(ii)] There exist positive constants $\ga'=\ga'(a,b)$ and $\gb'=\gb'(a,b)$ such that
 $$
  x^{\ga'}\leq {\mathrm{MAL}_H^{nu}(x)}\leq {x^{\gb'}},~\text{for}~x\gg 0.
 $$
\end{itemize}
\end{thm}

Some yet unproved number-theoretic conjectures imply that a polynomial upper bound is true also in the first case (see \cite{Be}).

Theorem \ref{thm:2.1} was proved in \cite{Be} for higher absolute rank groups but
the cases of $\PSL_2(\BR)$ and $\PSL_2(\BC)$ which are the most crucial for us were left open;
in particular, Theorem \ref{thm:2.1} answers a question from \cite{Be}.

The general strategy of the proof of  Theorem \ref{thm:counting-maximal}
is similar to \cite{Be} but some special considerations are needed for
small rank. (Note that in \cite{Be} the proof is easier for very high
rank - see Proposition 3.3  there). The seminal work of Borel \cite{Bo}
which gives a detailed description of the maximal arithmetic lattices
in $\PGL_2(\BR)^a\times\PGL_2(\BC)^b$ combined with some ideas of
Chinburg and Friedman \cite{CF} enables us to
prove it. Various number theoretic estimates are
needed along the way.

The fact that the number of maximal arithmetic lattices grows slowly
reduces the problem to the subgroup growth of a given such maximal
lattice $\Gamma$.

Recall now that for any finitely generated group $\G$, we have
$s_n(\G) \le (n!)^{d(\G)}$, where $s_n(\G)$ denotes the number of
subgroups of index at most $n$ in $\G$. This, combined with
Theorems \ref{thm:d(Gamma)} and \ref{thm:counting-maximal}, proves Theorem
\ref{thm:upperbound}. The proof of the precise bound in
Theorem \ref{thm:SL(2,R)} requires more.

For $\PSL_2(\BR)$ the "miracle" is that
the covolume of $\gC$ and the subgroup growth of $\gC$ are both
controlled by $\chi(\gC)$ -- the Euler characteristic of $\gC$.
The covolume is $-2\pi\chi(\Gamma)$ by the Gauss-Bonnet formula, and the number $s_n(\Gamma)$
of subgroups of index at most $n$ in $\Gamma$ is $n^{(-\chi(\Gamma) + o(1))n}$,
as was proved by Liebeck and Shalev \cite{LiSh}.
Thus the contribution of $\Gamma$ to $\mathrm{AL}_H(x)$ is roughly

\[
s_{\frac{x}{-2\pi\chi(\Gamma)}}(\Gamma)={(\frac{x}{-2\pi\chi(\Gamma)})}^{(-\chi(\Gamma)\frac{1}{-2\pi\chi(\Gamma)}+o(1))x} = x^{(\frac{1}{2\pi} + o(1))x},
\]
which, on the face of it, proves Theorem \ref{thm:SL(2,R)}.
However, there is a delicate point here: the behavior of the error
term $o(1)$ above depends on $\Gamma$.
This can be a serious problem: the issue is illustrated in \cite{BL} where
it is shown, in contrary to a conjecture from
\cite{BGLM} and \cite{GLNP}, that for high rank Lie groups the growth of the total
number of arithmetic lattices is strictly faster
than those arising from finite index subgroups of a given lattice.
Our next result yields a uniform bound on the subgroup growth of Fuchsian groups
and overcomes this difficulty:

\begin {thm}\label{thm:nsubgroups}\label{uniform}
There exists an absolute constant $c$
such that for every Fuchsian group $\Gamma$ and for every $n\in\BN$
we have
$$
s_n(\Gamma)\leq (cn)^{-\chi(\Gamma)n}.
$$
\end{thm}

The proof is a modification of the one given in \cite{LiSh}, it relies
heavily on bounds for the character values of symmetric groups.

Theorems \ref{thm:2.1} and \ref{thm:nsubgroups} imply the upper bound
in Theorem \ref{thm:SL(2,R)} (see Section \ref{sec:proof}). The lower bounds in
Theorems \ref{thm:lowerbound} and \ref{thm:SL(2,R)} are proved by analyzing
the subgroup growth of specific lattices. There is one delicate point
which has to be considered here: finite index subgroups of $\G$ may be conjugate
in $H$ without being conjugate in $\G$. An argument which uses the congruence
topology of $\G$ solves this problem and provides the lower bounds in all cases
(see Section \ref{sec:proof}).

\medskip
\noindent {\it Note added in proof:} Recently A. Eisenmann \cite{E} proved an analogous result for $H=\PSL_2(k)$, where $k$ is a $p$-adic field. His result says that if $k$ does not contain $\zeta +\zeta^{-1}$ (where $\zeta$ is the $p$'th root of unity) then $\lim\frac{\text{AL}_H(x)}{x\log x}=q-1$, where $q$ is the order of the residue field of $k$ (here $\mu$ is normalized to give value $1$ for the maximal compact subgroup of $H$.)

%------------------------------------------------------------------------------------------------------------------

\section{On the number of generators of a lattice}\label{sec:d(Gamma)}

In this section we give a proof of Theorem \ref{thm:d(Gamma)}.
Since the center of $H$ is finitely generated, replacing $H$ by its adjoint group we may assume it is center free.
We will assume below that $H$ is not locally isomorphic to $\PSL_2(\BR)$. For $H=\PSL_2(\BR)$ the theorem follows easily from the Gauss--Bonnet formula and the explicit presentation of lattices there (cf. Section \ref{sec:Fushcian}).

Let $K<H$ be a maximal compact subgroup and $X=H/K$ the associated rank $1$ Riemanian symmetric space. Then $H$ is a connected component of the group of isometries of $X$, and by our assumption $\dim X\ge 3$.
For $g\in H,~x\in X$ we denote by
$$
 d_g(x)=d(x,g\cdot x)
$$
the displacement of $g$ at $x$, and by
$$
 \Min(g):=\{x\in X:d_g(x)=\inf(d_g)\}
$$
the (possibly empty) set where $d_g$ attains its infimum.
Recall that there are $3$ types of isometries $g$ of $X$:
\begin{itemize}
\item $g$ is  {\it  elliptic} if it has a fixed point in $X$. In this case the set of fixed points $\mathrm{Fix}(g)=\Min(g)$ is a totally geodesic submanifold.
\item $g$ is {\it hyperbolic} if $\Min(g)\ne\emptyset$ but $\mathrm{Fix}(g)=\emptyset$. In this case $\Min(g)$ is an infinite two sided geodesic, it is called the {\it axis} of $g$.
\item $g$ is {\it parabolic} if $\Min(g)=\emptyset$. In this case $g$ has a unique fixed point $p$ at the visual boundary $\partial X$ of $X$, $\inf(d_g)=0$ and a geodesic $c:\BR\to X$ satisfies $c(\infty)=p$ if and only if $\lim_{t\to\infty}d_g(c(t))=0$.
\end{itemize}
Moreover, if $g_1,\ldots,g_k\in H$ are commuting elements which are simultaneously elliptic (resp. hyperbolic, resp. parabolic), then they have a common fixed point (resp. axis, resp. fixed point at $\partial X$).

Recall the classical Margulis lemma (cf. \cite[Chapter 4]{Th}):

\begin{lem}\label{lem:margulis}
There is a constant $\gep_H>0$, depending on $H$, such that if $\gL < H$ is a discrete group generated by $\{\gc\in\gL:d_\gc(x)\le\gep_H\}$ for some $x\in X$, then $\gL$ virtually nilpotent.
\end{lem}

Fix once and for all
$$
 \gep\le\min\{\frac{\gep_H}{10},1\}.
$$

Let $\gC$ be a lattice in $H$. Denote by $M=\gC\backslash X$ the corresponding orbifold, and by $\pi: X\to M$ the canonical (ramified) covering map. For a subset $Y\subset M$ let $\ti Y=\pi^{-1}(Y)$ be its preimage in $X=\ti M$.

Let
$$
 \ti N:=\cup\{ \Min(\gc):\gc\in\gC\setminus\{1\},~\inf d_{\gc}<\gep\}.
$$

Since $\gC$ acts properly discontinuously on $X$, $\ti N$ is a locally finite union of the sets $\Min (\gc)$.
Note that since $H$ is connected, any $g\in H$ preserves the orientation of $X$, and in particular if $g$ is elliptic $\mathrm{codim}_X(\mathrm{Fix}(g))\ge 2$. As $\ti N$ is a union of geodesics (axes of hyperbolic elements) and fixed sets of elliptic elements, the assumption that $\dim X\ge 3$ implies that $\mathrm{codim}_X(\ti N)\ge 2$. It follows that $X\setminus\ti N$ is connected.

For $\gc\in\gC\setminus\{ 1\}$ and $x\in X\setminus\ti N$ set
$$
 d_\gc'(x)=d_\gc(x) -\inf d_\gc.
$$

Let $f:\BR^{>0}\to\BR^{\ge 0}$ be a smooth function which tends to $\infty$ at $0$, strictly decreases on $(0,\gep]$ and is identically $0$ on $[\gep,\infty)$, and set for $x\in X\setminus\ti N$
$$
 \ti\psi(x):=\sum_{\gc\in\gC\setminus\{ 1\},~\inf d_\gc\le\gep}f(d'_\gc(x)).
$$
since $f(d'_\gc(x))\ne 0\Rightarrow {d_\gc}(x)\le 2\gep$ and $\gC$ is discrete, there are only finitely many non-zero summands for each $x$.
Thus $\ti\psi$ is a well defined $\gC$ invariant smooth function on $X\setminus \ti N$. Note that $\ti\psi(x)$ tends to $\infty$ as $x$ approaches $\ti N$. Let $\psi:M\setminus N\to\BR^{\ge 0}$ be the induced function on $M\setminus N$.

For $a\ge 0$, set
$$
 \ti \psi_{\le a}:=\{x\in X\setminus \ti N: \ti\psi(x)\le a\},
$$
and $\psi_{\le a}=\pi(\ti \psi_{\le a})=\{x\in X\setminus \ti N: \psi(x)\le a\}$. Note that $d_\gc(x)\ge\gep$ for every $x\in \ti \psi_{\le 0}$ and $\gc\in \gC\setminus\{ 1\}$. Thus, the injectivity radius of $M$ at any point in $\psi_{\le 0}$ is at least $\frac{\gep}{2}$.

\begin{lem}\label{lem:gradient}
For $x\in X\setminus\ti N$, the gradient $\nabla \ti\psi(x)=0$ if and only if $\ti\psi(x)=0$.
\end{lem}

The proof relies on the following simple observations which follow directly from the fact that $X$ has strictly negative sectional curvature:
\begin{enumerate}
\item If $c_1(t),c_2(t)$ are two disjoint unit speed geodesics with $c_1(-\infty)=c_2(-\infty)$, then $d(c_1(t),c_2(t))$ is a strictly increasing smooth function of $t$, and hence $\frac{d}{dt}d(c_1(t),c_2(t))>0$ for every $t$. In particular, if $g\cdot p=p$ for some $g\in H,p\in\partial X$ and $c$ is a geodesic with $c(-\infty)=p$, then $\frac{d}{dt}d_g(c(t))>0$ for all $t$, provided $c$ is not $g$ invariant.

\item Suppose that $A\subset X$ is a closed convex set and let $P_A:X\to A$ be the nearest point retraction. Then $d(P_A(x),P_A(y))<d(x,y)$ for any $x,y\in X\setminus A$. We will say that a geodesic ray $c:[0,\infty)\to X$ starting at a point of $A$ is {\it perpendicular} to $A$ if $P_A(c([0,\infty))=c(0)$. Then if $c_1,c_2:[0,\infty)\to X$ are two different unit speed rays perpendicular to $A$, the function $d(c_1(t),c_2(t))$ is convex and strictly increasing. Hence if $c:[0,\infty)\to X$ is perpendicular to $A$ and $g$ is an isometry which leaves the set $A$ invariant but moves $c(t)$, then $\frac{d}{dt}d_g(c(t))>0$ for any $t>0$.
\end{enumerate}

\noindent Note that $(1)$ and $(2)$ fail to hold in higher rank symmetric spaces. In rank one both assertions follow for instance from the well known Flat Strip Theorem: Two parallel segments in a Hadamard manifold bound a flat parallelogram, in particular, if the space has strictly negative curvature, then any two parallel segments are contained in a common geodesic (cf. \cite[Corollary 5.8]{Bal}).

\begin{proof}[Proof of Lemma \ref{lem:gradient}]
Let $x$ be a point in $X\setminus (\ti N\cup\ti\psi_{\le 0})$. Let $\gS_x$ be the finite set of elements in $\gC\setminus\{1\}$ which contribute non-zero summands to the function $\ti\psi$, i.e. those $\gs$ for which $f(d'_\gs(x))>0$.
We will show that there is a geodesic $c$ through $x$ (say, $x=c(t_0)$) such that $\frac{d}{dt}|_{t_0}d_\gs(c(t))>0,~\forall\gs\in\gS_x$. This will imply that
$$
 \nabla\ti\psi (x)\cdot\dot c(t_0)=\sum_{\gs\in\gS_x}f'(d'_\gs(x))\frac{d}{dt}|_{t_0}d_\gs(c(t))<0,
$$
and in particular that $\nabla\ti\psi (x)\ne 0$.

Let $\gD_x=\langle\gS_x\rangle$. By Lemma \ref{lem:margulis}, $\gD_x$ admits a normal nilpotent subgroup $\gL_x$ of finite index in $\gD_x$.

Suppose first that $\gD_x$ is a finite group. Let $y$ be a common fixed point of $\gS_x$. Then as $d_\gs(x)>0=d_\gs(y)$ for all $\gs\in\gS_x$, we derive from observation $(2)$ above (by considering $\{ y\}$ as a closed convex set) that the geodesic pointing from $y$ to $x$ does the job.

Suppose now that $\gD_x$ is infinite. Let $\gc$ be a nontrivial central element of the nilpotent group $\gL_x$, and let $C_{\gD_x}(\gc)$ be the conjugacy class of $\gc$ in $\gD_x$. Then $|C_{\gD_x}(\gc)|\le [\gD_x:\gL_x]$, $C_{\gD_x}(\gc)$ is contained in the center of $\gL_x$ and consists of elements which are simultaneously elliptic, hyperbolic or parabolic. In the first case, let $A$ be the set of common fixed points of the elements of $C_{\gD_x}(\gc)$, in the second case let $A$ be the common axis, and in the third case let $p\in\partial X$ be the common fixed point at infinity. Note that in the second case, as the geodesic $A$ is $\gD_x$-invariant, we derive that $\gD_x$ consists of semisimple elements and $A=\Min(\ga)=\mathrm{axis}(\ga)$ for every $\ga\in\gD_x$ which is hyperbolic.
It follows that $A$ does not contain our point $x$, for otherwise $\gS_x$ must consist of elliptic elements preserving $A$ and as $\gD_x$ is infinite, there are $\ga_1,\ga_2\in\gS_x$ without a common fixed point. However, since the $\ga_i$ preserve $A$ and $d_{\ga_i}(x)<\gep$, the element $\ga_1\ga_2$ is hyperbolic with axis $A$ and displacement $<\gep$, a contradiction (either to the assumption that $x\in A$ or to the one that $\gS_x$ has no hyperbolic elements).
In the first two cases we can take $c$ to be the geodesic through $x$ perpendicular to the $\gD_x$-invariant closed convex set $A$, and in the third case we can take $c$ to be the geodesic through $x$ with $c(-\infty)=p$. The result follows from observations $(1)$ and $(2)$ above.
\end{proof}

Since $\lim_{t\to 0}f(t)=\infty$, it follows that for any finite value $a$, the injectivity radius is bounded from below on the closed set $\psi_{\le a}$. Hence $\psi_{\le a}$ is bounded, for otherwise it would admit an infinite $1$-discrete subset, yielding infinitely many disjoint embedded balls of a fixed radius in $M$, contradicting the finiteness of $\vol(M)$. Therefore $\psi_{\le a}$ is compact for any $a<\infty$. Thus the function $\psi$ is proper.
Applying standard Morse theory we get:

\begin{prop}\label{prop:retraction}
For every positive $a$, $\psi_{\le a}$ is a deformation retract of $M\setminus N$.
\end{prop}

\begin{proof}
By Lemma  \ref{lem:gradient} the proper smooth function $\psi:M\setminus N:\to\BR^{\ge 0}$ has no positive critical values. Thus the proposition follows from
\cite[Theorem 3.1]{Milnor}.
\end{proof}

For $a>0$, we will denote by $r_a:M\setminus N\to\psi_{\le a}$ the retraction induced by the deformation retract supplied by Proposition \ref{prop:retraction}.

\begin{cor}
For every $a\ge 0$ the set $\ti\psi_{\le a}$ is non empty and connected.
\end{cor}

\begin{proof}
For $a>0$ this immediately follows from Proposition \ref{prop:retraction}, and for $a=0$ it follows since $\ti\psi_{\le 0}=\bigcap\{\ti\psi_{\le a}:{a>0}\}$ a descending intersection of the sets $\psi_{\le a}$.
\end{proof}

Since $\ti\psi_{\le a}$ is $\gC$--invariant and $\gC$ acts freely on it, we conclude:

\begin{cor}
For every $a>0$, $\gC$ is a quotient of $\pi_1(\psi_{\le a})$ -- the fundamental group of $\psi_{\le a}=\gC\backslash\ti\psi_{\le a}$.
\end{cor}

Now recall that the injectivity radius of $M$ at any point of $\psi_{\le 0}$ is at least $\gep/2$. Let $\mathcal{S}$ be a maximal $\gep/2$ discrete subset of $\psi_{\le 0}$. For $t>0$ denote by $\nu (t)$ the volume of a $t$--ball in $X$.
Since the $\gep/4$--balls centered at points of $\mathcal{S}$ are pairwise disjoint and isometric to an $\gep/4$--ball in $X$, the size of $\mathcal{S}$ is bounded by $\vol (M)/\nu(\frac{\gep}{4})$. Moreover, since $\mathcal{S}$ is maximal, the union of the $\gep/2$--balls centered at points of $\mathcal{S}$ covers $\psi_{\le 0}$. Denote this union by $U$.
Since $U$ is a neighborhood of the compact set $\psi_{\le 0}$ and $\psi$ is continuous, choosing $a_0>0$ sufficiently small, we have  $\psi_{\le a_0}\subset U$.

\begin{lem}
$\pi_1( \psi_{\le a_0})$ is a quotient of $\pi_1(U)$.
\end{lem}

\begin{proof}
The inclusion $i:\psi_{\le a_0}\to U$ induces a map $i^*:\pi_1(\psi_{\le a_0})\to\pi_1(U)$, and the retraction $r_{a_0}$ restricted to $U$, $r_{a_0}:U\to \psi_{\le a_0}$ induces a map $r_{a_0}^*:\pi_1(U)\to\pi_1(\psi_{\le a_0})$. Since $r_{a_0}\circ i$ is the identity on $\psi_{\le a_0}$, we see that $r_{a_0}^*\circ i^*$ is the identity on $\pi_1(\psi_{\le a_0})$. It follows that $r_{a_0}^*:\pi_1(U)\to\pi_1(\psi_{\le a_0})$ is onto.
\end{proof}

Since $M$ is negatively curved, the $\gep/2$--balls centered at points of $\mathcal{S}$ are convex, and hence any non-empty intersections of such is convex, hence contractible. Thus these balls form a {\it good cover} of $U$, in the sense of \cite{BoTu} (see also \cite{Ge}), and the nerve $\mathcal{N}$ of this cover is homotopic to $U$. Now $\pi_1(U)\cong\pi_1(\mathcal{N})$ has a generating set of size $\le E(\mathcal{N})$ -- the number of edges of the 1--skeleton $\mathcal{N}^1$. To see this one may choose a spanning tree $\mathcal{T}$ for the graph $\mathcal{N}^1$ and pick one generator for each edge belonging to $\mathcal{N}^1\setminus\mathcal{T}$.
Finally note that the edges of $\mathcal{N}$ correspond to pairs of points in $\mathcal{S}$ which are of distance at most $\gep$. Thus the degrees of the vertices  of $\mathcal{N}$ are uniformly bounded, in fact, one can show that $\nu(1.25\gep)/\nu(0.25\gep)$ is an upper bound for the degrees. Thus
$$
 d(\gC)\le |E(\mathcal{N})|\le \nu(1.25\gep)/\nu(0.25\gep)\cdot |\mathcal{S}|
 \le\frac{\nu(1.25\gep)}{\nu(0.25\gep)^2}\cdot\vol(M).
$$

\begin{rem}\label{rem:d(Gamma)}
It has been recently shown that Theorem \ref{thm:d(Gamma)} holds for all semisimple Lie groups, with no rank assumption (see \cite{Ge:morse}). \end{rem}

%-------------------------------------------------------------

\section{Counting maximal arithmetic subgroups}\label{sec:maximal}

The main goal of this section is to prove Theorem \ref{thm:counting-maximal}.
Only the upper bound is needed for the main result of this paper
(and a much weaker estimate suffices).

For simplicity of notations we will assume throughout the proof that $H=\PGL_2(\BR)$
and remark at various points if non-trivial modifications are needed for the general case.
(For a proof of an even more general result including products of $\PGL_2(k)$, where $k$ is any characteristic $0$ local field, see the forthcoming thesis \cite{E}.) 
We normalize the Haar measure on $H$ so that for a lattice $\Gamma$ its covolume
is equal to the hyperbolic volume of the corresponding locally symmetric space.
Borel \cite{Bo} described in detail the maximal
arithmetic lattices in $H$. We will follow the
exposition of his work in \cite{MR}.

Arithmetic lattices in $H$ are all
obtained in the following way: Let $k$ be a totally real number
field of degree $d=d_k$, and $A$ a quaternion algebra over $k$. Assume
that for all the archimedean valuations $v\in V_\infty$ except for a single one $v_0$, $A$ ramifies over $k_v (\cong \BR)$, i.e., $A(k_v)$
is isomorphic to the Hamiltonian quaternion algebra, while
$A(k_{v_0})\cong {\mathrm M}_2(\BR)$. Let $\OO=\OO_k$ be the ring of
integers in $k$ and let $\mathfrak{D}$ be an order in $A(k)$, i.e.,
$\mathfrak{D}$ is a finitely generated $\OO$-submodule of $A(k)$
which is also a subring that generates $A(k)$ over $k$.
Furthermore, assume that $\mathfrak{D}$ is a maximal order in $A$. Let
$\mathfrak{D}^*$ be the group of invertible elements of $\mathfrak{D}$. Now,
$\mathfrak{D}^*$ is discrete in $\prod_{v\in V_{\infty}}
(A(k_v))^*\cong \mathrm{U}(2)^{d-1}\times \GL_2(\BR)$, and its
projection $\gC$ to $\PGL_2(\BR)$ gives a discrete subgroup of
$\PGL_2(\BR)$. A subgroup of $H=\PGL_2(\BR)$ which is
commensurable with such $\Gamma$ is called an {\it arithmetic subgroup}
of $H$. It is always a lattice, i.e., a discrete subgroup of finite covolume in
$H$. It is non-uniform if and only if $k=\mathbb{Q}$ and A is the split
quaternion algebra ${\mathrm M}_2(\mathbb{Q})$, i.e., $\Gamma$ is commensurable
to $\PGL_2(\mathbb{Z})$. All the arithmetic lattices of
$\PGL_2(\BR)$ are obtained in this way and two arithmetic
lattices of $H$ are commensurable if and only if they come, in the
process as above, from the same algebra $A$. So for each such algebra
$A$ we can associate a well defined commensurability class of
arithmetic lattices in $H$ which is denoted by $\mathcal{C}(A)$.

Now, a quaternion algebra $A$ over $k$ is completely determined by the
finite set of valuations $\ram(A)$, a subset of the set of all
valuations $V$ of $k$ which consists of those $v$ for which $A(k_v)$
ramifies (i.e., $A(k_v)$ is a division algebra), while for $v\in
V\setminus \ram(A)$, $A(k_v)$ splits (namely, isomorphic to
${\mathrm M}_2(k_v)$). The subset $\ram(A)$ must be of even size; in our case
it is formed by exactly $d-1$ real valuations and a subset $\ram_f(A)$,
possibly empty, of non-archimedean valuations. The set $\ram_f(A)$
can be identified with a subset of the prime ideals of $\OO$.
Let $\Delta (A)= \prod_{\mathcal{P}\in \ram_f(A)}\mathcal{P}$,
the product of all prime ideals at which $A$ ramifies.

Borel showed that the commensurability class $\mathcal{C}(A)$ has
infinitely many non-conjugate maximal elements but only finitely many of them have bounded covolume. The minimal covolume in the class
$\mathcal{C}(A)$ is:
\begin{equation}
 \frac{8\pi
 {\Delta_k}^{3/2}\zeta_k(2)\prod_{\mathcal{P}|\Delta(A)}(N(\mathcal{P})-1)}{(4{\pi}^2)^ {d_k}  [{R_{f,\infty}}^*:({R_f}^*)^2]
 [{}_2J_1:J_2]},\label{formula:maxcovolume}
\end{equation}
see \cite[Corollary 11.6.6, p. 361 and (11.6), p. 333]{MR}. (Note that in \cite[(11.28), p. 361]{MR} the formula is given for $\PGL_2(\BC)$, in which case the factor $8\pi$ in the nominator is replaced by $4\pi^2$.)

Here $\Delta_k$ is the absolute value of the discriminant of $k$, $\zeta _k$ is the
Dedekind zeta function of $k$ and $N(P)$ denotes the norm of the ideal $P$, i.e.
the order of the quotient field $\OO /P$. For the other
notations we refer to \cite[p. 358]{MR}. For our purpose it is enough to know the following estimates:

\begin{equation}\label{formula:2.3}
 1\leq [{R_{f,\infty}}^*:({R_f}^*)^2][{}_2J_1:J_2]\leq2^{d+\mid
 \ram_f(A)\mid}h_k,
\end{equation}
 where $h_k$ is the class number of $k$ (see \cite[pp. 358--360]{MR}).

\medskip

We can now begin with the proof of Theorem \ref{thm:counting-maximal}. Assume that $x$ is a large real number. We first bound the number of possible fields $k$ which can contribute a maximal arithmetic lattice of  covolume at most $x$. Then, given $k$, we will bound the number of possible quaternion algebras $A$ and finally, given $k$ and $A$, we will estimate the number of conjugacy classes of maximal lattices in $\mathcal{C}(A)$ of covolume at most $x$.

\begin{lem}\label{lem:1} There exist two constants $c_1$ and $c_2$ such that if for some quaternion algebra $A$ over $k$ as above $\mathcal{C}(A)$ contains a lattice of covolume at most x in $\PGL_2(\BR)$, then $d_k \leq c_1 \log x + c_2$.
\end{lem}

\begin{proof}
By \cite[Lemma 4.3]{CF}, we know that if $\Gamma \in \mathcal{C}(A)$, then
\begin{equation}\label{formula:2.5}
 \covol(\Gamma ) > 0.69 \exp (0.37d_k - \frac{19.08}{h(k,2,A)}),
\end{equation}
where $h(k,2,A)$ is the order of a certain quotient of the class group of $k$, in particular, $1\leq h(k,2,A)\leq h_k$.
This lemma is one of the main technical results of \cite{CF}, its proof uses a variety of number-theoretic techniques.

We obtain
\begin{equation}
 0.69\exp(0.37d_k - 19.08) \leq x,
\end{equation}
and so
\begin{equation}\label{formula:2.6}
 d_k \leq 3\log x + 21.
\end{equation}
\end{proof}

\begin{lem} \label{lemma:2.2}
There exist constants $c_3, c_4 \in \BR^{>0}$ such that if for some $A$ and $k$ as above there is $\Gamma$ in $\mathcal{C}(A)$ of covolume $ \leq x$, then $\gD_k\leq c_3x^{c_4}$.
\end{lem}

\begin{proof}
Borel and Prasad \cite[(7) on p. 143]{BP} deduced from Brauer-Siegel Theorem and a result of Zimmert that
\begin{equation}\label{formula:2.9}
 h_k \leq 10^2(\frac{\pi}{12})^{d_k}\Delta_k.
\end{equation}
Combining (\ref{formula:maxcovolume}), (\ref{formula:2.3}) and (\ref{formula:2.9}), we obtain
\begin{equation}\label{formula:2.10}
 \frac{8\pi{\Delta_k}^{3/2}\zeta_k(2)\prod_{P\in \ram_f(A)}(N(P)-1)}{(4\pi^2)^{d_k}2^{d_k +|\ram_f(A)|}10^2(\frac{\pi}{12})^{d_k}\Delta_k}\leq x.
\end{equation}

Now $\frac{N(P)-1}{2}\geq 1$ unless $|N(P)|=2$, which can happen for at most $d_k$ primes in $\ram_f(A)$ and for those $\frac{N(P)-1}{2}=\frac{1}{2}$. Also,  $\zeta_k(2)\geq 1$. Thus from (\ref{formula:2.10}) we deduce
\begin{equation}\label{formula:2.12}
x\ge
\frac{8\pi{\Delta_k}^{1/2}}{10^2(4\pi^2)^{d_k}2^{2d_k}(\frac{\pi}{12})^{d_k}}= \frac{8\pi}{10^2}\frac{\Delta_k^{\frac{1}{2}}}{(\frac{4}{3}\pi^3)^{d_k}}
\end{equation}

Now, use (\ref{formula:2.6}) to deduce that
$$
 \frac{8\pi}{10^2}\Delta_k^{\frac{1}{2}}\leq x(\frac{4}{3}\pi^3)^{3\log x+21},
$$
which implies the lemma.
\end{proof}

We can now appeal to a theorem of Ellenberg and Venkatesh (\cite{EV}, \cite[Appendix]{Be}):

\begin{thm}\label{thm:EV}
\begin{enumerate}
\item [(i)] Let $N(x)$ denote the number of isomorphism classes of number fields $k$ with $\gD_k\le x$. Then for every $\epsilon > 0$, there is a constant $c_5(\epsilon)$ such that $\log N(x) \leq c_5(\epsilon)(\log x)^{1+\epsilon}$ for every $x\geq 2$.
\item
[(ii)] For a fixed $d$, let $N_d(x)$ be the number of isomorphism classes of number fields $k$ of degree at most $d$ with $\gD_k\le x$.
Then there exist $c_5'=c_5'(d)$ and $c_5''=c_5''(d)$ such that $N_d(x)\le c_5'\cdot x^{c_5''}$ for all sufficiently large $x$.
\end{enumerate}
\end{thm}

We can deduce from Lemma \ref{lemma:2.2} and  Theorem \ref{thm:EV} that for every $\epsilon > 0$, the number of number fields $k$ which contribute lattices of covolume at most $x$ is bounded by $x^{c_6(\epsilon)(\log x)^\epsilon}$ for a constant $c_6(\epsilon)$ and $x\gg 0$.
Given one of these fields $k$ with degree $d_k$, we now estimate the number of relevant $A$'s. This number is bounded from above by the number of possible ways to choose an even set of valuations $\ram(A)$ consisting of $d_k-1$ real valuations and $\ram_f(A)$, which satisfy (\ref{formula:2.10}).
There are $d_k$ ways to choose the $d_k-1$ real valuations. Now, $8\pi\Delta_k^{\frac{1}{2}}\geq 1$ and also $\zeta_k(2)\geq1$, by Lemma \ref{lem:1}, $d_k\leq c_1\log x+c_2$, and hence by (\ref{formula:2.10}) we get
\begin{equation}\label{formula:2.15}
 \frac{1}{2^{\mid \ram_f(A)\mid}}\prod_{P\in \ram_f(A)}(N(P)-1)\leq c_7x^{c_8},
\end{equation}
for some absolute constants $c_7$ and $c_8$.
The number of primes $P$ with $N(P)=2$ or 3 is at most $2d_k$. For them $\frac{N(P)-1}{2}\geq\frac{1}{2}$ and for all the other primes $\frac{N(P)-1}{2}\geq N(P)^{\log_4(\frac{3}{2})}$.
We can therefore deduce that:
\begin{equation}\label{formula:2.16}
 \prod_{P\in \ram_f(A)}N(P)\leq x^{c_9},
\end{equation}
for some absolute constant $c_9$.

Note that $\Delta(A)=\prod_{P\in \ram_f(A)}P$ is a square free ideal of norm at most
$x^{c_9}$, whose factors determine $A$ modulo the $d_k$ choices of the unique real valuation of $k$ in which $A$ splits.

\begin{lem} \label{lemma:2.17}
Let $I_k(x)$ denote the number of ideals of $\OO_k$ of norm less than $x$. Then $I_k(x)\leq\zeta_k(2)x^2\leq(\frac{\pi^2}{6})^{d_k}x^2$.
\end{lem}

\begin{proof}
$I_k(x)= a_1+a_2+\ldots+a_{[x]}$ where $a_n$ is the number of ideals of norm $n$. At the same time $\zeta_k(s)=\sum_{n=1}^{\infty}a_nn^{-s}$. Hence, for every large $x$, $\zeta_k(2)x^2\geq I_k(x)$. It is easy to see that $\zeta_k(2)\leq\zeta_\mathbb{Q}(2)^{d_k}$ and it is well known that $\zeta_\mathbb{Q}(2)=\frac{\pi^2}{6}$.
\end{proof}

Again, as $d_k = O(\log x)$, we can deduce from Lemma \ref{lemma:2.17} and (\ref{formula:2.16}) that given $k$, the number of possibilities for $\Delta(A)$ is polynomial in $x$.
So all together we now have a bound of the form $x^{c(\epsilon)(\log x)^{\epsilon}}$ for the number of quaternion algebras which give rise to lattices of covolume $\leq x$, or, in other words, for the number of commensurability classes $\mathcal{C}(A)$ with representatives of covolume at most $x$.
Before continuing, let us mention that if one is interested in non-uniform lattices there is at most one such $A$ (i.e., $k=\mathbb{Q}$ and $A={\mathrm M}_2(\mathbb{Q})$) and exactly one when $x$ is large enough. This is the case for $\PGL_2(\BR)$. For general $H=\PGL_2(\BR)^a\times\PGL_2(\BC)^b$ this number is polynomial by Theorem \ref{thm:EV} (ii) since the degree of the field of definition of a non-uniform lattice in such $H$ is $a+2b$.

Now fix $k$ and $A$, or, equivalently, $\ram(A)$. We need to count the maximal lattices within the class $\mathcal{C}(A)$. In \cite[Section 11.4]{MR} a class $m(A)$ of lattices in $\mathcal{C}(A)$ is described which essentially gives the maximal subgroups in $\mathcal{C}(A)$, namely, it contains all the maximal ones but maybe some more. We will show that even the total number of all of them is polynomial in $x$. An element of $m(A)$ is denoted by $\Gamma_{S,\mathfrak{D}}$ where $\mathfrak{D}$ is a maximal order in $A$ (in \cite{MR} it is denoted by $\OO$ but we prefer to use $\mathfrak{D}$ as $\OO$ for us is the ring of integers in $k$) and $S$ is a finite set of finite primes of $k$ disjoint from $\ram_f(A)$. When $S$ is the empty set we get the group $\Gamma_{\emptyset,\mathfrak{D}}$. This is a group of the minimal covolume in $\mathcal{C}(A)$ whose covolume is given by (\ref{formula:maxcovolume}). Up to conjugacy, the number of such groups is the same as the number of conjugacy classes of maximal orders $\mathfrak{D}$ in $A$. This number is called the type number of $A$ (cf. \cite[Section 6.7]{MR}). The type number is a power of $2$ \cite[Corollary 6.7.7]{MR}) which divides $h_k2^{d_k-1}$ (see \cite[eq. (6.13) p. 221]{MR}). Recall that $d_k=O(\log x)$ (by Lemma \ref{lem:1}). On the other hand by (\ref{formula:2.9}), $h_k \leq 10^2(\frac{\pi}{12})^{d_k}\Delta_k$ and by Lemma \ref{lemma:2.2}, $\Delta_k$ is polynomially bounded in $x$. Thus, the type number of $A$ is also polynomially bounded in $x$. We can therefore fix $\mathfrak{D}$ and count the number of possibilities for $S$. The exact form of $\Gamma_{S,\mathfrak{D}}$ is described in \cite[Section 11.4]{MR}, but for our current purpose what is only relevant is its covolume which is given by Theorem 11.5.1 on page 357 there:

\begin{prop} \label{proposition:2.19}
$\begin{displaystyle}
\frac{\covol(\Gamma_{S,\mathfrak{D}})}{\covol(\Gamma_{\emptyset,\mathfrak{D}})}= 2^{-m}\prod_{P\in S}(N(P)+1)
\end{displaystyle}$ for some $0\leq m\leq |S|$.
\end{prop}

Recall that the covolume of $\Gamma_{\emptyset,\mathfrak{D}}$ is given by (\ref{formula:maxcovolume}), but at this point we do not need it. Just recall that a well known result of Siegel asserts that for all lattices $\Gamma$ in $\PGL_2(\BR)$, $\covol(\Gamma)\geq \frac{\pi}{42}$. (Note that in both cases $G=\PSL_2(\BR)$ and $G=\PGL_2(\BR)$ we have used the measure induced from the hyperbolic structure on $G/K=\BH^2$. Now the lattice of minimal covolume in $\PGL_2(\BR)$ has as fundamental domain the pull back of the triangle $(\frac{\pi}{2},\frac{\pi}{3},\frac{\pi}{7})$ while its intersection with $\PSL_2(\BR)$ is of minimal covolume there but needs two copies of that triangle.) We can deduce now that if $\covol(\Gamma_{S,\mathfrak{D}})\leq x$, then
\[\label{formula:2-20} \prod_{P\in S}\frac{N(P)+1}{2}\leq\frac{42}{\pi}x.
\]
Arguing exactly as we did in (\ref{formula:2.15}), (\ref{formula:2.16}) and Lemma \ref{lemma:2.17} when bounding the possibilities for $\ram_f(A)$ we deduce that the number of possibilities for $S$ is polynomial in $x$ (in fact, it is even easier now as we do not need to exclude the primes $2$ and $3$).

If $H = \PGL_2(\BR)^a\times\PGL_2(\BC)^b$, instead of Siegel's theorem we can use the Kazhdan--Margulis theorem \cite{KM}, or Borel's result \cite[Theorem 8.2]{Bo} which implies that covolumes of arithmetic lattices in $H$ are bounded from bellow by a positive constant which depends only on $H$.
This finishes the proof of the upper bounds in both parts of \ref{thm:counting-maximal}.

\medskip

The proof of the lower bounds is much easier but it requires a detailed description of the groups $\Gamma_{S,\mathfrak{D}}$ above. As the lower bounds are not really needed for the main results, we only sketch the argument assuming that the reader is familiar with Section 11.4 of \cite{MR}.

For the lower bound in case (a) of Theorem \ref{thm:counting-maximal} we could vary the algebra $A$ and easily deduce that for some $\delta > 0$, there are at least $x^\delta$ non conjugate (and also not commensurable after conjugation) maximal arithmetic lattices in $\PGL_2(\BR)$ with covolume at most $x$, but such a proof would not work for the non-uniform case where all arithmetic lattices are commensurable (after conjugation) and there is only one $A={\mathrm M}_2(\BQ)$. For a proof which works in both cases fix $k$, $A$ and $\mathfrak{D}$ so that $\mathcal{C}(A)$ is a commensurability class of arithmetic lattices in $H$. We need to show that there exist sufficiently many subsets $S$ with $\prod_{v\in S}q_v\le x^c$ and $\Gamma_{S,\mathfrak{D}}$ is maximal. For the group $\Gamma_{S,\mathfrak{D}}$ to be maximal there should exist an element $a\in k$ such that $\tilde{v}(a)$ is odd for $v\in S$ (here $\tilde{v}(a)$ denotes the logarithmic valuation on $k_v$), $a$ is positive at the ramified real places of $A$,  and $\tilde{v}(a)$ is even for $v\in V_f \setminus (\ram_f(A)\cup S)$ (see discussion in Section 11.4 of \cite{MR} for more details).

Let $p$ be a rational prime and let $S_p$ be the set of places $v\in V_f$ such that $\tilde{v}(p)$ is odd. Assume that $p$ is unramified in $k$ which is the case for all sufficiently large primes. Then $S_p$ consists of the prime ideals of $k$ which divide $p$, so if $p_1 \neq p_2$, then $S_{p_1} \neq S_{p_2}$. To every such $p$ we can assign a maximal arithmetic subgroup $\Gamma_p = \Gamma_{S_p,\mathfrak{D}}$ which contains an element odd at $v$ for $v\in S_p$ and even at the remaining places in $V_f\setminus \ram_f(A)$. Then for $p_1 \neq p_2$ the groups $\Gamma_{p_1}$ and $\Gamma_{p_2}$ are non-conjugate maximal arithmetic subgroups of $H$. By Borel's volume formula (see (\ref{formula:maxcovolume}) and Proposition \ref{proposition:2.19})
$$\covol(\Gamma_p) \le c_1 \prod_{v\in S_p}(q_v +1),$$
where $c_1 = c_1(k,A)$ is a positive constant.

The set $S_p$ contains at most $d_k$ places of $k$ and $\prod_{v\in {S_p}}q_v\le p^{d_k}$. This implies
$$\covol(\Gamma_p) \le c_1 (2p)^{d_k}.$$
Thus, if $p \le \frac12(x/c_1)^{1/d_k}= x^{c_2}$, then $\covol(\Gamma_p) \le x$. As $k$ is fixed and so is $d_k$, it follows from the prime number theorem that
for a large enough $x$ there exists a constant $\gd>0$ such that there are at least $x^\gd$ such primes $p$, and hence such $\Gamma_p$'s.

This finishes the proof of the theorem. \qed

\begin{rem}
Theorem \ref{thm:counting-maximal} gives a solution to Problem~6.5 of \cite{Be} and implies that we can now remove the restriction
that the group $H$ has no simple factors of type $\mathrm{A}_1$ in the main result there.
\end{rem}

\begin{rem}
With the estimate of Lemma \ref{lemma:2.2} at hand we can essentially repeat all the steps of the proof of Theorem A in \cite{BP} for the groups $\mathrm{G}$ defined over number fields and having the absolute rank 1. The only missing ingredient, which is an analogue of a number-theoretic result from Section 6.1 in \cite{BP} for the groups of type ${\mathrm A}_1$, is now available because of Lemmas \ref{lem:1} and \ref{lemma:2.2}. This allows us to remove the restriction on the absolute rank from the statement of the Borel-Prasad theorem, which can be now formulated as follows:

\begin{thm}
Let $c > 0$ be given. Assume $k$ runs through the number fields. Then there are only finitely many choices of $k$, of an absolutely almost simple algebraic group $\mathrm{G}'$ defined over $k$ up to $k$-isomorphism, of a finite set $S$ of places of $k$ containing all the archimedean places, of arithmetic subgroup $\Gamma'$ of $\mathrm{G}'_S = \prod_{v\in S} \mathrm{G}'(k_v)$ up to conjugacy, such that $\mu'_S(\mathrm{G}'_S/\Gamma') \le c$.
\end{thm}

We refer to \cite{BP} for the definition of the measure $\mu'_S$ with respect to which the covolumes of arithmetic subgroups are computed (the so-called {\em Tits measure}), introduction and more details on this important result. Note that an analogous result for the groups over the global function fields is not true. For example, one can show that the covolume of $\SL_n(\mathrm{F}_q[t^{-1}])$ in $\SL_n(\mathrm{F}_q((t)))$ tends to zero if either $n$ or $q\to\infty$ (see \cite[Section 7.12]{BP}).
\end{rem}

%-----------------------------------------------------------------------------------

\section{Uniform bounds for Fuchsian groups and characters of Symmetric groups}\label{sec:Fushcian}

In this section we study the subgroup growth of finitely generated
non-elementary discrete subgroups of $\PGL_2(\BR)$, the so called
Fuchsian groups.

By classical work of
Fricke and Klein, the orientation-preserving Fuchsian groups $\G$
(i.e. those contained in $\PSL_2(\BR)$) have a
presentation of the following form:
\begin{equation}
\begin{array}{lll}
\mathrm{generators:} &~& a_1,b_1,\ldots ,a_g,b_g,~\mathrm{(hyperbolic)},~
                     x_1,\ldots ,x_d ~ \mathrm{(elliptic)},\\
                    &~& y_1,\ldots ,y_s  ~ \mathrm{(parabolic)},~
                    z_1,\ldots ,z_t  ~\mathrm{(hyperbolic~boundary~elements)}\\ \\

\mathrm{relations: } &~& x_1^{m_1} = \cdots = x_d^{m_d} = 1,~
                    x_1\cdots x_d \,y_1 \cdots y_s \,z_1 \cdots z_t\, [a_1,b_1]\cdots
                   [a_g,b_g] = 1,
\end{array}
\end{equation}

\noindent
where $g,d,s,t \geq 0$ and $m_i \geq 2$ for all $i$.
The number $g$ is referred to as the {\it genus} of $\G$.
Define $\mu(\G) = -\c(\G)$, where $\c(\G)$ is the Euler characteristic
of $\G$ (in the rest of this section we reserve the letter
$\c$ for characters of symmetric groups).
Then for $\G$ as above we have
\[
\mu(\G) = 2g-2 + \sum_{i=1}^d (1 - \frac{1}{m_i}) + s+t.
\]
It is well known that $\mu(\G) > 0$.

The non orientation-preserving Fuchsian groups have
presentations as follows, with $g>0$:
\begin{equation}
\begin{array}{lll}
\mathrm{generators: } &~& a_1,\ldots ,a_g,~
                    x_1,\ldots ,x_d,~
                    y_1,\ldots ,y_s,~
                    z_1,\ldots ,z_t \\ \\
\mathrm{relations: } &~& x_1^{m_1} = \cdots = x_d^{m_d} = 1,  \\
                   &~& x_1\cdots x_d\, y_1 \cdots y_s \,z_1 \cdots z_t \,a_1^2\cdots
                   a_g^2 = 1.
\end{array}
\end{equation}

In this case we have
\[
\mu(\G) = -\c(\G) = g-2 + \sum_{i=1}^d (1 - \frac{1}{m_i}) + s+t,
\]
and again, $\mu(\G) >0$.

We call Fuchsian groups as in (11) {\it oriented}, and those as in (12)
{\it non-oriented}.

The Fuchsian groups with $s=t=0$ are the uniform lattices;
these are more challenging ones since the other Fuchsian groups
are free products of cyclic groups.

In this section we prove Theorem \ref{uniform} which provides a uniform
bound on the subgroup growth of Fuchsian groups.

The novelty of Theorem \ref{uniform} is that it holds for {\it all} $n$
(not just for large $n$, where large may depend on $\G$), and
that it uses $\mu(\G)$ as the only parameter.
For a fixed group there are more refined asymptotic results
(see \cite{LiSh}), but it is the uniform version above which
is crucial for our applications.

We shall now embark on the proof of Theorem~\ref{uniform}.
We assume throughout this section that $\G$ is an oriented Fuchsian
group with the presentation given in (11). The proof in the
non-oriented case (12) is very similar, hence omitted.

A major tool in our proof is character theory of symmetric groups
(see \cite{Sa}).
We begin with some relevant notation and results.

Denote by $\Irr(S_n)$ the set of all irreducible characters of $S_n$.
By a partition of a positive integer $n$ we mean a tuple
$\la = (\la_1,\ldots ,\la_r)$ with $\la_1 \geq
\la_2 \geq \ldots \geq \la_r\geq 1$ and $\sum_{i=1}^r \la_i = n$.
Denote by $\chi_\la$ the irreducible
character of $S_n$ corresponding to the partition $\la$.

The following result of Fomin and Lulov \cite{FL} plays
a key role in our proof.

\begin{prop}\label{FL}
{\rm (Fomin-Lulov \cite{FL})} Fix an integer $m \geq 2$. Suppose
$n$ is divisible by $m$, say $n = am$, and let $\pi \in S_n$ be a permutation
of cycle-shape $(m^a)$. Then for any irreducible character $\chi$ of $S_n$,
we have
\[
|\chi(\pi)| \leq \frac{a!\,m^a}{(n!)^{1/m}}\cdot \chi(1)^{1/m}.
\]
Consequently we have
\[
|\chi(\pi)| \le b n^{{1/2}} \cdot \chi(1)^{1/m},
\]
where $b$ is some absolute constant.
\end{prop}

We note that the second assertion follows from first using
Stirling's formula. In fact, as noted in \cite{FL}, we even have
$|\chi(\pi)| \le b n^{1/2-1/(2m)} \cdot \chi(1)^{1/m}$.
In the rest of this section $b$ denotes the constant above.

We shall also frequently use the Murnaghan-Nakayama Rule \cite[p. 180]{Sa}.
By a {\it rim $r$-hook} $\nu$ in a $\la$-tableau, we mean a connected part of
the tableau containing $r$ boxes, obtained by starting from a box at the
right end of a row and
at each step moving downwards or leftwards only, which can be removed
to leave a proper tableau denoted by $\la \backslash \nu$.
If, moving from right to left,
the rim hook $\nu$
starts in row $i$ and finishes in column $j$, then the {\it leg-length}
$l(\nu)$ is defined to be the number of boxes below the $ij$-box
in the $\la$-tableau.

\begin{lem}\label{murn} {\rm (Murnaghan-Nakayama Rule)} Let $\rho \s \in S_n$, where
$\s$ is an $r$-cycle and $\rho$ is a permutation of the remaining $n-r$
points. Then
\[
\chi_\la (\rho \s) = \sum_\nu (-1)^{l(\nu)} \chi_{\la \backslash \nu} (\rho),
\]
where the sum is over all rim $r$-hooks $\nu$ in a $\la$-tableau.
\end{lem}

In order to apply the Murnaghan-Nakayama rule it is useful to estimate
the number of rim $r$-hooks in a tableau. We quote Lemma 2.11
from \cite{LiSh}:

\begin{lem}\label{amazing} For any positive integer $r$ and any partition
$\la$ of $n$, the number of rim $r$-hooks in a $\la$-tableau is at most
$\sqrt{2n}$.
\end{lem}

We now deduce the following.

\begin{lem}\label{values}

Let $\pi = \rho \s \in S_n$ be a permutation of order $m$,
where $\rho$ has cycle-shape $(m^a)$ and $\s$ permutes the remaining $n-ma$
points. Let $C(\s)$ be the number of cycles in $\s$. Then for any $\chi
\in \Irr(S_n)$ we have
\[
|\chi(\pi)| \leq b(2n)^{C(\s)/2} \chi(1)^{1/m}n^{1/2}.
\]
\end{lem}

\begin{proof}
Applying the Murnaghan-Nakayama Rule repeatedly for each cycle in $\s$
and Lemma~\ref{amazing}, we see that
\[
|\chi(\pi)| \leq \sum |\chi_i(\rho)|,
\]
where $\c_i \in \Irr(S_{ma})$,
the sum has at most $(2n)^{C(\s)/2}$ terms, and $\chi_i(1) \leq
\chi(1)$. By Proposition~\ref{FL},
$|\chi_i(\rho)| \leq b\chi_i(1)^{1/m}n^{1/2} \le b\chi(1)^{1/m}n^{1/2}$.
This implies
\[
|\chi(\pi)| \leq b(2n)^{C(\s)/2}\chi(1)^{1/m}n^{1/2}
\]
as required.
\end{proof}

For a permutation $\pi \in S_n$ we let $\pi^{S_n}$ denote its
conjugacy class.
We now deduce the following.

\begin{prop}\label{T}
Let $m \geq 2$ be an integer.
Then for any positive integer $n$, any permutation $\pi \in S_n$
satisfying $\pi^m = 1$ and any character $\c \in \Irr(S_n)$, we have
\[
|\pi^{S_n}|\cdot |\c(\pi)| < b (n^n)^{1-1/m} \cdot \c(1)^{1/m}\cdot
(2e)^n.
\]
\end{prop}

\begin{proof}
Let $\pi$ have cycle-shape $(m_1^{a_1},\ldots ,m_k^{a_k})$, where $\sum
m_ia_i = n$, $m_1>m_2> \ldots > m_k$ and $m_1=m$ (allowing the possibility
that $a_1 = 0$). Set $A = \sum_{i=2}^k a_i$. Since
\[
|\pi^{S_n}| = \frac{n!}{\prod_{i=1}^k m_i^{a_i} \prod_{i=1}^k a_i!}
\le  \frac{n!}{\prod_{i=1}^k a_i!},
\]
Lemma~\ref{values} implies that
\begin{equation}\label{b1}
|\pi^{S_n}| \cdot |\c(\pi)| \leq \frac{n!}{\prod_{i=1}^k a_i!}
(2n)^{A/2} b \chi(1)^{1/m} n^{1/2}.
\end{equation}

Since $\frac{n^a}{a!} < e^a$ for any positive integer $a$, we have
$\frac{n^{\sum a_i}}{\prod a_i!} < e^{\sum a_i}$. Let
$B = \sum_{i=1}^k a_i$. Then
\[
\frac{n!}{\prod_{i=1}^k a_i!} \le \frac{n^n}{\prod_{i=1}^k a_i!}
= \frac{n^B}{\prod_{i=1}^k a_i!} n^{n-B} < e^B n^{n-B}.
\]

In view of (\ref{b1}) this implies
\[
|\pi^{S_n}|\cdot |\c(\pi)| <  e^B n^{n-B}
(2n)^{A/2} b \chi(1)^{1/m} n^{1/2} \le b n^{n-B+A/2} \c(1)^{1/m}
e^B 2^{A/2} n^{1/2}.
\]
Now, for $i \ge 2$ we have
$m_i \le m/2$ (since $m_i$ is a proper divisor of $m$), so
$a_1m + A m/2 \ge \sum_{i=1}^k a_i m_i = n$. Thus
\[
B - A/2 = a_1 + A/2 \ge n/m.
\]
We also have $e^B 2^{A/2} n^{1/2} \le e^n 2^{n/2} 2^{n/2} = (2e)^n$.
Hence
\[
|\pi^{S_n}|\cdot |\c(\pi)| < b n^{n-n/m} \c(1)^{1/m} (2e)^n,
\]
which proves the result.
\end{proof}

We shall also use the following.

\begin{lem}\label{chardegs} Fix a real number $s>0$. Then
\[
\sum_{\chi \in \Irr(S_n)} \chi(1)^{-s} \go 2 \;as\; n \go \infty.
\]
In particular, $\sum_{\chi \in \Irr(S_n)} \chi(1)^{-s} \le C(s)$
for all $n$, where $C(s)$ is a number depending only on $s$.
\end{lem}

This result is proved in \cite{LiSh} (see Theorem 1.1 there)
following earlier results
for the case $s \ge 1$. Here we apply it for
a rather small value of $s$, namely $s = \frac{1}{42}$,
which is the minimal value of $\mu(\G)$ for a Fuchsian group $\G$.

Next we study the space of homomorphisms ${\rm Hom}(\G,S_n)$
from a Fuchsian group $\G$ to $S_n$ by splitting it into subspaces
whose sizes can be estimated.
Let $g, d, s, t, m_1, \ldots , m_d$ be as in (11) and let
$\mu = \mu(\G)$.

Let $g_1, \ldots , g_d \in S_n$ be permutations satisfying
$g_i^{m_i} = 1$ for each $i$.
Let $C_i = g_i^{S_n}$ ($1 \leq i \leq d$) be their conjugacy classes
in $S_n$. Write ${\bf C} = (C_1,\ldots ,C_d)$. Define
\[
{\rm Hom}_{\bf C}(\G,S_n) = \{ \phi \in {\rm Hom}(\G,S_n)\;:\; \phi(x_i)
\in C_i \hbox{ for }1\leq i \leq d\}.
\]

Suppose now that $\G$ is a uniform lattice.
The following formula, which essentially dates back to Hurwitz \cite{Hu},
connects $|{\rm Hom}_{\bf C}(\G,S_n)|$ with characters
of symmetric groups:
\begin{equation}\label{b2}
|{\rm Hom}_{\bf C}(\G,S_n)| =
(n!)^{2g-1} {|C_1|\ldots |C_d|} \sum_{\c \in \Irr(S_n)}
\frac{\c(g_1)\cdots \c(g_d)}{\chi(1)^{d-2+2g}}.
\end{equation}
This formula includes the case $d=0$ in which $\G$ is a surface group,
${\rm Hom}_{\bf C} = {\rm Hom}$ and empty products are taken to be 1.

In fact, formula (\ref{b2}) holds for any finite group $G$ in place of $S_n$,
its proof is carried out by counting solutions in $G$ of the equations
corresponding to the defining relations of $\G$. See, for instance,
Section 3 of \cite{LiSh} for details and for a similar formula for
non-oriented Fuchsian groups.

\begin{lem}\label{C} With the above notation we have
\[
|{\rm Hom}_{\bf C}(\G,S_n)| \le c_1 b^d (n^n)^{\mu + 1} (2e)^{dn},
\]
where $c_1$ is some absolute constant.
\end{lem}

\begin{proof}
Suppose first that $\G$ is uniform. We use formula (\ref{b2}) above.
By Proposition~\ref{T}, we have
\[
|C_i| |\c(g_i)| \le b(2e)^n (n^n)^{1 - 1/m_i} \c(1)^{1/m_i}
\]
for all $i = 1, \cdots , d$.
This yields
\[
|{\rm Hom}_{\bf C}(\G,S_n)| \le
(n^n)^{2g-1}b^d(2e)^{dn} (n^n)^{\sum_{i=1}^d (1 - 1/m_i)}
\sum_{\c \in \Irr(S_n)}
\frac{\c(1)^{\sum_{i=1}^d 1/m_i}}{\c(1)^{d-2+2g} }.
\]
Since $\mu = 2g-2 + \sum_{i=1}^d (1-1/m_i)$, we conclude that
\[
|{\rm Hom}_{\bf C}(\G,S_n)| \le
b^d(2e)^{dn} (n^n)^{\mu + 1} \sum_{\c \in \Irr(S_n)} \c(1)^{-\mu}.
\]
By Lemma~\ref{chardegs} and a well known Siegel's inequality $\mu \ge 1/42$,
we have
\[
\sum_{\c \in \Irr(S_n)} \c(1)^{-\mu} \le \sum_{\c \in \Irr(S_n)} \c(1)^{-1/42}
\le c_1,
\]
where $c_1 = C(1/42)$.
The result follows by combining the two inequalities above.

It remains to deal with Fuchsian groups with $s+t > 0$.
Let $r = 2g + s + t -1$, $Z_m$ denote
a cyclic group of order $m$ and $F_r$ a free group of rank $r$.
Then we have a free product decomposition
\[
\G \cong Z_{m_1} * \cdots * Z_{m_d} * F_r.
\]
It follows immediately that
\begin{equation}\label{b3}
|{\rm Hom}_{\bf C}(\G,S_n)| = (n!)^r \prod_{i=1}^d |C_i|.
\end{equation}

We claim that if $\pi \in S_n$ satisfies $\pi^m = 1$, then
\begin{equation}\label{b4}
|\pi^{S_n}| \le e^n (n^n)^{1-1/m}.
\end{equation}
Indeed, let $a_1, \ldots , a_k$ be the multiplicities of the
cycle lengths of $\pi$ and let $B = \sum_{i=1}^k a_i$,
the number of cycles in $\pi$. Then, as in the proof of
Proposition~\ref{T}, we have
\[
|\pi^{S_n}| \le \frac{n!}{\prod a_i!} \le e^B n^{n-B}.
\]
Since all cycle lengths of $\pi$ are at most $m$, we have
$n/m \le B \le n$. Thus $e^B n^{n-B} \le e^n n^{n-n/m}$,
proving the claim.

Recall that $C_i = g_i^{S_n}$ and $g_i^{m_i} = 1$.
Therefore, using the claim, we obtain
\[
|C_i| \le e^n (n^n)^{1-1/m_i}
\]
for all $i = 1, \ldots , d$.
Plugging this in (\ref{b3}) we conclude that
\[
|{\rm Hom}_{\bf C}(\G,S_n)| = (n!)^r e^{dn} (n^n)^{\sum (1-1/m_i)}
\le (n^n)^{r + \sum (1-1/m_i)} e^{dn}.
\]
Since $r + \sum (1-1/m_i) = \mu+1$ the result follows
(with an even better upper bound).
\end{proof}

We can now provide upper bounds for $|{\rm Hom}(\G,S_n)|$.
Given $d$ define $d_1 = \max (d,1)$.

\begin{lem}\label{upper} There exists an absolute constant $c_2$ such
that with the above notation we have
\[
|{\rm Hom}(\G,S_n)| \leq (n^n)^{\mu + 1} c_2^{d_1 n}
\]
for all $n$.
\end{lem}

\begin{proof} We clearly have
\[
|{\rm Hom}(\G,S_n)| = \sum_{\bf C} |{\rm Hom}_{\bf C}(\G,S_n)|,
\]
where the sum is over all possible ${\bf C} = (C_1, \dots , C_d)$.
Now, $S_n$ has $p(n)$ conjugacy classes, where $p(n)$ is the
partition function, and each summand $|{\rm Hom}_{\bf C}(\G,S_n)|$
can be bounded as in Lemma~\ref{C}. This yields
\[
|{\rm Hom}(\G,S_n)| \le p(n)^d \cdot  c_1 b^d (n^n)^{\mu + 1} (2e)^{dn}.
\]
It is well known that $p(n) \le c_3^{\sqrt{n}} \le c_3^n$ for some
absolute constant $c_3$. This yields
\[
|{\rm Hom}(\G,S_n)| \le c_1 (n^n)^{\mu + 1} c_4^{dn},
\]
for the absolute constant $c_4 = 2e b c_3$.
This easily implies the required conclusion (with $c_2 = c_1c_4$).
\end{proof}

We can now draw conclusions to the subgroup growth of Fuchsian groups.

\begin{prop}\label{main} There exists an absolute constant $c_5$ such
that with the above notation we have
\[
s_n(\G) \leq n^{\mu n} c_5^{d_1 n}
\]
for all $n$.
\end{prop}

\begin{proof}
For a positive integer $n$, denote by $a_n(\G)$ the number of index
$n$ subgroups of $\G$.
Define
\[
{\rm Hom}_{trans}(\G,S_n) = \{ \phi \in {\rm Hom}(\G,S_n) : \phi(\G)
\hbox{ is transitive}\}.
\]
It is well known that $a_n(\G) = |{\rm Hom}_{trans}(\G,S_n)|/(n-1)!$
(see, for instance, \cite[1.1.1]{LS}).
Obviously, $|{\rm Hom}_{trans}(\G,S_n)| \le |{\rm Hom}(\G,S_n)|$,
so applying Lemma~\ref{upper} we obtain
\[
a_n(\G) \le  (n^n)^{\mu + 1} c_2^{d_1 n} / (n-1)!,
\]
and hence
\[
s_n(\G) \le  n \cdot (n^n)^{\mu + 1} c_2^{d_1 n} / (n-1)!.
\]
Since $n \cdot n^n / (n-1)! \le c_6^n$ for some absolute
constant $c_6$, we obtain
\[
s_n(\G) \le  c_6^n \cdot (n^n)^{\mu} c_2^{d_1 n},
\]
which implies the conclusion (with $c_5 = c_2 c_6$).
\end{proof}

We can finally prove the main result of this section, namely
Theorem~\ref{uniform}.

Note that $\mu = 2g-2 + \sum_{i=1}^d (1-1/m_i) \ge -2 + d/2$.
This yields $d \le 2\mu +4$, and so
\[
d_1 \le d+1 \le 2 \mu + 5.
\]
Next, since $\mu \ge 1/42$, we have $5 \le 210 \mu$, and so
$2 \mu + 5 \le 212 \mu$. This implies
\[
d_1 \le 212 \mu.
\]
Applying Proposition~\ref{main}, we obtain
\[
s_n(\G) \leq n^{\mu n} c_5^{212 \mu n} \le (cn)^{\mu n},
\]
where $c = c_5^{212}$.

This completes the proof of Theorem~\ref{uniform}. \qed

\medskip

We did not make an attempt
to optimize the constants appearing here and in Theorem~\ref{uniform}
in particular.

\begin{rem}
It is intriguing that there are no uniform
lower bounds on the subgroup growth of Fuchsian groups. For example,
let $\G$ be a triangle $(p,q,r)$-group with $p,q,r$ distinct
primes which are greater than $n$ (so $\G$ has a presentation as in (11)
with $g=s=t=0$, $d=3$ and $m_1=p, m_2 = q, m_3 = r$).
Then one easily sees that $s_n(\G)=0$. Since $n$ can be arbitrarily
large no uniform lower bounds exist.
\end{rem}

Finally, we can provide a lower bound on $s_n(\G)$ in terms of $\mu(\G)$
which holds for all $n$ larger than some number depending on $\G$:

\begin{prop}\label{prop:4.10} Let $\G$ be a Fuchsian group. Then we have
\[
s_n(\G) \ge (n!)^{\mu(\G)} \hbox{ for all sufficiently large } n.
\]
\end{prop}

\begin{proof} A somewhat stronger result, namely
$a_n(\G) \ge (n!)^{\mu(\G)} \cdot n^b$
for every fixed $b$ and for all large $n$ is given in Theorem 4.6
of \cite{LiSh}. However, the proof given there only works for Fuchsian
groups with torsion. The remaining Fuchsian groups are free groups and
surface groups.
For the free group $\G = F_d$ ($d \ge 2$) it is well known that
$a_n(\G) \sim (n!)^{d-1} \cdot n = (n!)^{\mu(\G)} \cdot n$.
For surface groups $\G$ it is known that
$a_n(\G) \sim 2 (n!)^{\mu(\G)}  \cdot n$ (see \cite{MP}).
This implies the result. In fact, it follows that for any constant
$c<1$ and for any Fuchsian group $\G$ we have
\[
a_n(\G) \ge cn (n!)^{\mu(\G)},  \hbox{ for all sufficiently large } n,
\]
and this lower bound is best possible.
\end{proof}

%------------------------------------------------------------------------------------

\section {Proof of the main results}\label{sec:proof}

We can use now the results of Sections \ref{sec:d(Gamma)}, \ref{sec:maximal} and \ref{sec:Fushcian} to prove the main Theorems.

\subsection{The proof of Theorem \ref{thm:upperbound}.} Recall from \cite{LS}, Lemma 1.1.2 that if $\gC$ is finitely generated group with $d(\gC)$ generators, then $s_n(\gC)\le n^{d(\gC)n}$ where $s_n(\gC)$ denotes the number of subgroups of index at most $n$ in $\gC$. Let now $H$ be a fixed rank one simple Lie group. By Theorem \ref{thm:counting-maximal} and \cite{Be}, $\mathrm{MAL}_H(x) = \mathrm{MAL}^{u}_H(x) + \mathrm{MAL}^{nu}_H(x)\le x^{a\log x}$. Now, every arithmetic lattice $\gC$ of covolume at most $x$ is contained in a maximal lattice $\gL$ of covolume, say, $y\le x$, and $[\gL:\gC]\le\frac{x}{y}$. Hence,
$$
 \mathrm{AL}_H(x)\le\mathrm{MAL}_H(x)\cdot\max\{s_{\frac{x}{y}}(\gL):y\le    x,\mathrm{covol}(\gL)=y\}
$$
where $\gL$ runs over the finitely many conjugacy classes of arithmetic lattices of covolume at most $x$. By Theorem \ref{thm:d(Gamma)}, for such $\gL$, $d(\gL)\le Cy$ for some constant $C=C(H)$, hence
$$
 s_\frac{x}{y}(\gL)\le(\frac{x}{y})^{d(\gL)\frac{x}{y}}\le(\frac{x}{y})^{Cy\frac{x}{y}}.
$$
Now by Kazhdan-Margulis theorem (which also follows from Theorem \ref{thm:d(Gamma)} as we noted in the introduction), $y\ge C'$ for some $C'=C'(H)$, and hence
$$
 \mathrm{AL}_H(x)\le x^{a\log x}(\frac{x}{C'})^{Cx}\le x^{bx}
$$
for some constant $b$ when $x\gg 0$.
\qed

\begin{rem}\label{rem:5.1}
As mentioned in Remark \ref{rem:d(Gamma)}, Theorem \ref{thm:d(Gamma)} holds also for higher rank semisimple Lie groups, thus the proof shows that Theorem \ref {thm:upperbound} is valid in general. We refer to \cite{BL} for better estimates in the higher rank case.
\end{rem}

%------------------------------------------------------------

\subsection{The proof of Theorem \ref{thm:lowerbound}.}
Fix $n\ge 2$. In \cite{Lub2} it was shown that $H=\PSO(n,1)$ has an arithmetic lattice $\gC$ which is mapped onto the free group on two generators. Let $\mathrm{L}(n)$ be the set of subgroups of $\gC$ of index at most $n$. Then
$$
 |\mathrm{L}(n)|=s_n(\gC)\ge s_n(F_2)\ge n!
$$
(see \cite[Corollary 2.1.2]{LS} for the right hand side inequality).
Thus, if $\mathrm{covol}(\gC)=v_0$, $\gC$ contains at least $n!$ sublattices of covolume at most $nv_0$. We should prove that not too many of them are conjugate in $H$.

To prove this we argue as follows: Let $C(n)$ be the set of congruence subgroups of $\gC$ of index at most $n$. By \cite{Lub1}, $|C(n)|\le n^{c_0\log n/\log\log n}$ for some constant $c_0$. There is a map $\phi:\mathrm{L}(n)\to C(n)$ sending a subgroup in $\mathrm{L}(n)$ to its closure in the congruence topology of $\gC$. There are therefore at least $\frac{n!}{n^{c_0\log n/\log\log n}}\ge n^{c_1n}$ subgroups in $\mathrm{L}(n)$ with the same congruence closure. If $\gC_1$ and $\gC_2$ are two such subgroups with congruence closure $\gC_0$ which are conjugate in $H$, say $\gC_1=g\gC_2 g^{-1}$, then $g\in\mathrm{Comm}_H(\gC)$, the commensurability group of $\gC$ in $H$. As $\gC$ is arithmetic, $\mathrm{Comm}_H(\gC)$ acts by $k$-rational morphisms (more precisely, in the adjoint representation it has entries in $k$, where $k$ is the number field over which $\gC$ is defined), hence $g$ preserves the congruence topology. Thus, $g$ normalizes $\gC_0$, i.e. $g\in N_H(\gC_0)$. Recall that $N_H(\gC_0)$ is itself a lattice in $H$ (cf. \cite{Raghunathan}), hence by Kazhdan-Margulis theorem (or by \ref{thm:d(Gamma)}) has covolume at least $C'$. This implies that at most $c_2n$ subgroups like $\gC_2$ can be conjugated to $\gC_1$ in $H$. Thus, $H$ has at least $n^{c_1n}/c_2n\ge n^{c_1'n}$ conjugacy classes of arithmetic lattices of covolume at most $nv_0$. This proves Theorem \ref{thm:lowerbound}.

%----------------------------------------------------------------

\subsection{The proof of Theorem \ref{thm:SL(2,R)}}
We shall prove a more precise statement:

\begin{thm} \label{theorem:4.1}
Let $H=\PSL_2(\BR)$ and $\mathrm{AL}_H(x)$ the number of conjugacy classes of arithmetic lattices in $H$ of covolume at most $x$. Then there exist $6<r<s\in \BR$ such that
$$
 (rx)^{\frac{1}{2\pi}x} \leq \mathrm{AL}_H(x)\leq (sx)^{\frac{1}{2\pi}x}
$$
for all sufficiently large $x$.
\end{thm}

\begin{proof}
We will prove first the upper bound. As in the proof of Theorem \ref{thm:upperbound} above, $\mathrm{AL}_H(x)\le\mathrm{MAL}_H(x)\cdot\max\{ s_{\frac{x}{y}}(\gL)\}$. Again, by Theorem \ref{thm:counting-maximal}, $\mathrm{MAL}_H(x)\le x^{b\log x}$, but this time we also know that $y=\mathrm{covol}(\gL)=-2\pi\chi(\gL)$ and by Theorem \ref{thm:nsubgroups}, $s_\frac{x}{y}(\gL)
\le(c\frac{x}{y})^{-\chi(\gL)\frac{x}{y}}$. Thus altogether
$$
 \mathrm{AL}_H(x)\le x^{b\log x}\cdot(c\frac{x}{y})^{-\chi(\gC)\frac{x}{-2\pi\chi (\gC)}}\le (sx)^{\frac{1}{2\pi}x}
$$ for a suitable constant $s$, as $y\ge\frac{\pi}{21}$ (by Siegel's theorem).

To prove the lower bound, take $\gC$ to be the
arithmetic lattice in $H=\PSL_2(\BR)$ of the lowest covolume ($\frac{\pi}{21}$), namely the $(2,3,7)$-triangle group.
By Proposition \ref{prop:4.10}, $s_n(\gC)\ge n^{-\chi(\gC)n}$ for $n\gg 0$. Again using the congruence topology of $\gC$ and the upper bound on the number of congruence subgroups from \cite{Lub1}, we can deduce, as in the proof of Theorem \ref{thm:lowerbound} above, that $\gC$ has at least
$n^{-\chi(\gC)n}/[(n^{c_0\log n/\log\log n})(c_2n)]\ge (c_3n)^{-\chi(\gC)n}$
subgroups of index at most $n$ which belong to different $H$-conjugacy classes, where the constant $c_3$ could be taken $1-\gep$ for arbitrary positive $\gep$. Each such subgroup is a lattice of covolume at most $x:=n\cdot\mathrm{covol}(\gC)=-2\pi\chi(\gC)n$. Thus we have found at least
$(c_3\frac{x}{\mathrm{covol}(\gC)})^{-\chi(\gC)\frac{x}{-2\pi\chi(\gC)}}\ge(\frac{c_3}{(\pi/21)}x)^{\frac{x}{2\pi}}$
conjugacy classes of arithmetic lattices in $H$ of covolume at most $x$. This finishes the proof of Theorem \ref{theorem:4.1} and consequently of \ref{thm:SL(2,R)}.
\end{proof}

Let us remark that the use of the $(2,3,7)$-triangle group was made only to ensure that $r>6$. For proving the lower bound of Theorem \ref{thm:SL(2,R)} we could use any arithmetic lattice in $\PSL_2(\BR)$.

\begin{proof}[Proof of Corollary \ref{cor:1.4}]
The genus $g$ surfaces are in one to one correspondence with the conjugacy classes of torsion free lattices in $\PSL_2(\BR)$ of covolume $4(g-1)\pi$. Thus Theorem \ref{thm:SL(2,R)} implies the upper bound in \ref{cor:1.4}. The proof of the lower bound is similar to the proof of the lower bound in \ref{thm:SL(2,R)}, just choose as a starting arithmetic lattice one which is {\it torsion free} and then count its finite index subgroups.
\end{proof}
%-----------------------------------------------------------------------------

\end{document}